\documentclass[a4paper,11pt]{article}

\usepackage[utf8]{inputenc}
\usepackage{amsfonts}
\usepackage{amssymb}
\usepackage{amsthm}
\usepackage{amsmath}
\usepackage{a4wide}
\usepackage{mathrsfs}
\usepackage{epsfig}
\usepackage{palatino}
\usepackage{tikz}
\usepackage{fp,ifthen}
\usepackage{nicefrac}
\usetikzlibrary{decorations.pathreplacing}
\usepackage{float}
\usepackage{enumerate} 
\usepackage{enumitem} 
\usepackage[a4paper,twoside,top=2cm, bottom=1.7cm, left=1.3cm, right=1.3cm]{geometry}

\usepackage[colorinlistoftodos,textwidth=2.3cm]{todonotes} 

\newcommand{\pdfgraphics}{\ifpdf\DeclareGraphicsExtensions{.pdf,.jpg}\else\fi}
\usepackage{graphicx}

\usepackage{color}
\definecolor{hanblue}{rgb}{0.27, 0.42, 0.81}
\definecolor{red}{rgb}{1.0, 0.0, 0.0}
\usepackage[colorlinks, citecolor=blue,linkcolor=blue, urlcolor = blue]{hyperref}
\usepackage[english,capitalize]{cleveref}

\usepackage{mathtools}

\theoremstyle{plain}

\newtheorem{teo}{Theorem}[section]
\newtheorem{lemma}[teo]{Lemma}
\newtheorem{prop}[teo]{Proposition}
\newtheorem{cor}[teo]{Corollary}

\theoremstyle{definition}
\newtheorem{defn}[teo]{Definition}

\theoremstyle{remark}
\newtheorem{rem}[teo]{Remark}

\numberwithin{equation}{section}

\newcommand{\de}{\ensuremath{\,\mathrm d}} 
\renewcommand{\d}{\ensuremath{\mathrm d}} 
\newcommand{\st}{\ensuremath{\ :\ }} 
\newcommand{\eqdef}{\ensuremath{\vcentcolon=}}
\newcommand \eps{\ensuremath{\varepsilon}} 
\renewcommand{\epsilon}{\varepsilon}
\newcommand{\N}{\ensuremath{\mathbb N}}
\newcommand{\R}{\ensuremath{\mathbb R}}

\newcommand{\TT}{\mathsf{T}}
\newcommand{\NN}{\mathsf{N}}
\DeclarePairedDelimiter\scal{\langle}{\rangle} 
\newcommand{\Ell}{\mathscr{L}}
\newcommand{\G}{\mathcal{G}} 
\renewcommand{\t}{\mathfrak{t}} 

\makeatletter
\renewcommand*\env@matrix[1][*\c@MaxMatrixCols c]{%
  \hskip -\arraycolsep
  \let\@ifnextchar\new@ifnextchar
  \array{#1}}
\makeatother

\begin{document}

\pdfgraphics 

\title{On the uniqueness of nondegenerate blowups for the motion by curvature of networks}

\author{Alessandra Pluda \footnote{\href{mailto:alessandra.pluda@unipi.it}{alessandra.pluda@unipi.it}, Dipartimento di Matematica, Universit\`a di Pisa, Largo Bruno Pontecorvo 5, 56127 Pisa, Italy.} \and Marco Pozzetta \footnote{\href{mailtto:marco.pozzetta@unina.it}{marco.pozzetta@unina.it}, Dipartimento di Matematica e Applicazioni, Universit\`a di Napoli Federico II, Via Cintia, Monte S. Angelo 80126 Napoli, Italy}}




\date{\today}

\maketitle

\vspace{-0.5cm}
\begin{abstract}
\noindent In this note we prove uniqueness of nondegenerate compact blowups for the motion by curvature of planar networks. The proof follows ideas introduced in \cite{PludaPozzMinimalNetworks} for the study of stability properties of critical points of the length functional and it is based on the application of a \L ojasiewicz--Simon gradient inequality. 
\end{abstract}




    
    

\section{Introduction}

If $G$ is an abstract one-dimensional graph, see \cref{def:Networks}, a network $\Gamma:G\to\R^2$ is, roughly speaking, an embedding of $G$ in $\R^2$ of class $H^2$. Topologically, we only consider graphs whose junctions have order three, and we will be mainly interested in regular networks, i.e., networks such that the tangent vectors of the immersions given by $\Gamma$ forms equal angles, equal to $\tfrac23\pi$, at any junction.

A motion by curvature $\Gamma_t:G\to \R^2$ is a smooth one-parameter family of networks such that the normal velocity of the evolution is equal to the curvature of the immersions along each edge, see \cref{def:Motion}.  The motion by curvature generalizes the one-dimensional mean curvature flow, called curve shortening flow, to singular one-dimensional objects given by planar networks.
The basic theory concerning short and long time behavior of this geometric flow can be found in \cite{Bronsardreitich, MantegazzaNovagaTortorelli, mannovplusch, GMP}. Just like classical mean curvature flow, the motion by curvature may develop singularities determined by the fact the either the length of some edge goes to zero or the $L^2$-norm of the curvature diverges in finite time. A great deal of work has been done to understand the nature of these singularities
and to define the flow past singularities~\cite{MantegazzaNovagaTortorelli,MaMaNo16,mannovplusch,BaldiHausMan,IlmNevSch,liramazplusae,MaNoPl21,Chang,Chang2}.

A classical procedure employed in order to analyze a singularity for a flow $\Gamma_t:G\to\R^2$ is to perform a \emph{blowup} of the motion. Classically, if $\Gamma_t$ is defined on times $[0,T)$ and $x_0\in \R^2$ is fixed, this means to consider sequences of parabolically rescaled motions by curvature $\Gamma^{\mu_j}_t\eqdef \mu_j \big( \Gamma_{-\mu_j^{-2}t+T}-x_0\big):G\to\R^2$ and to pass to the limit in space-time as $\mu_j\to+\infty$. Up to subsequence, one shows that there exists such a limit $\Gamma^\infty_t$, which is called tangent flow, and it is a homothetically shrinking motion by curvature, see \cite[Section 8]{mannovplusch}. The shape of $\Gamma^\infty_t$ at any time identifies the motion by curvature $\Gamma_t$ at infinitesimal scales around $x_0$ as $t$ approaches $T$. However, the construction may clearly depend on the choice of the sequence $\mu_j$, and it is today an open problem to prove uniqueness of blowups, i.e., independence of the sequence $\mu_j$, for a motion by curvature, or for the mean curvature flow, in full generality. Some of the most influential contributions to the uniqueness of blowups for mean curvature flow are contained in \cite{Schulze}, for the case of compact blowups, \cite{ColdingMinicozziUniqueness}, for the case of cylindrical blowups, \cite{ChodoshSchulze}, for the case of conical blowups.\\
An equivalent way for proving uniqueness of blowups, employed also in the above mentioned works, is to prove that the so-called rescaled motion by curvature $\widetilde\Gamma_\t\eqdef \big(2e^{-\t}\big)^{-\frac12}\big( \Gamma_{T-e^{-\t}} - x_0\big):G\to\R^2$ has a full limit as $\t\to+\infty$, see \cref{sec:Motion}. As in the case of tangent flows, along sequences $\t_n\to+\infty$, one always finds limits of $\widetilde\Gamma_{\t_n}$, but the limit may depend on the diverging sequence of times.

In this note, we prove that if one of such sequential limits $\Gamma_*$ of $\widetilde\Gamma_{\t_n}$ is nondegenerate and compact, encoded in the fact that $\Gamma_*:G\to \R^2$ is a regular network for the same graph $G$ which is the domain of the original flow $\Gamma_t:G\to\R^2$, then it is unique.
For the precise definitions of the terminology used in the next statement, we refer to \cref{def:Networks}, \cref{def:Motion}, \cref{def:RescaledMotion}, and \cref{rem:EquivalenzaTangentRescaled} below.

\begin{teo}\label{thm:Uniqueness}
Let $\Gamma_t:G\to \R^2$ be a motion by curvature defined on times $t\in[0,T)$ with $T<+\infty$, where $G$ is a graph without endpoints.

If $\Gamma_*:G\to \R^2$ is a nonempty nondegenerate blowup of $\Gamma_t$ at some $x_0 \in \R^2$ and time $T$, then $\Gamma_*$ is the unique blowup of $\Gamma_t$ at point $x_0$ and time $T$.

More precisely, the following two (equivalent) statements hold true.
\begin{enumerate}
    \item The rescaled motion by curvature $\widetilde\Gamma_t$ at point $x_0$ and time $T$ smoothly converges to $\Gamma_*$ as $t\to+\infty$, up to reparametrizations.
    
    \item For any diverging sequence $\mu_j\nearrow +\infty$, the sequence of parabolically rescaled motions by curvature $\Gamma^{\mu_j}_t$ converges locally smoothly in space-time to the homothetically shrinking motion by curvature $\Gamma^\infty_t$ such that $\Gamma^\infty_{-\frac12}=\Gamma_*$.
\end{enumerate}
\end{teo}
We remark that the above uniqueness theorem holds true regardless of assumptions on whether $T$ is a singularity time for $\Gamma_t$ of Type I or Type II, see \cite[remark 8.21]{mannovplusch}. On the other hand, a necessary consequence of the uniqueness result in \cref{thm:Uniqueness} is the fact that if $\Gamma_t$ develops a singularity at time $T$ and a blowup at $T$ is a nondegenerate network $\Gamma_*:G\to\R^2$, then such singularity is, in fact, a Type I singularity (see \cite[Remark 8.24]{mannovplusch}).

\cref{thm:Uniqueness} is a consequence of \cref{thm:Convergence} below, which proves smooth convergence of a rescaled motion $\widetilde\Gamma_\t$. The great advantage of working with the flow $\widetilde\Gamma_\t$ is that such evolution turns out to be the $L^2$-gradient flow of the \emph{Gaussian energy} defined by
\begin{equation*}
    \G(\Gamma)\eqdef \sum_i \int_0^1 e^{-|\gamma^i(x)|^2/2} \de s ,
\end{equation*}
for any network $\Gamma:G\to\R^2$.
With the obvious meaning of symbols $\sum_i \int_{\gamma^i}$ and $\int_\Gamma$, we shall equivalently write $\G(\Gamma)=\sum_i \int_{\gamma^i} e^{-|p|^2/2} = \int_\Gamma e^{-|p|^2/2}$.\\
The gradient flow structure permits to apply a general strategy for proving full convergence of a flow out of its sequential convergence along sequences of times, see \cite{PozzettaLoja}. Such strategy is based on the application of a \L ojasiewicz--Simon gradient inequality \cite{Si83, Ch03, ChFaSc09} derived on the energy functional $\G$. We derive such inequality in \cref{cor:LojaShrinker} below.

The main difficulty here is that such an inequality would readily follow by functional analytic results \cite{Ch03} relying on the possibility of parametrizing networks as normal graphs over fixed critical points of $\G$. However, the nonsmooth structure of networks necessarily implies technical complications, since networks close to a fixed one $\Gamma_*$ cannot be written as normal graphs over $\Gamma_*$. Hence, in order to perform a graph parametrization of networks close to $\Gamma_*$ we will need to allow for graphs having both a normal and a tangential component with respect to $\Gamma_*$. The issue of performing such a graph parametrization and at the same time of achieving the linear functional analytic setup needed for deriving the \L ojasiewicz--Simon inequality was investigated in \cite{PludaPozzMinimalNetworks} for the study of stability properties networks that are critical points of the length functional. We shall recall from \cite{PludaPozzMinimalNetworks} the tools needed here at the beginning of \cref{sec:Blowups}, and we refer to \cite{PludaPozzMinimalNetworks} for a more detailed account on the method.

\bigskip
\textbf{Organization.} In \cref{sec:Preliminaries} we recall basic definitions on networks, motion by curvature, and blowups. In \cref{sec:Blowups} we first derive the \L ojasiewicz--Simon inequality at critical points of $\G$, \cref{cor:LojaShrinker}, and then we prove convergence of the rescaled motion by curvature, \cref{thm:Convergence}, implying the main result \cref{thm:Uniqueness}.


\section{Preliminaries}\label{sec:Preliminaries}

\subsection{Networks}

For a regular curve $\gamma:[0,1]\to \R^2$ of class $H^2$, we define the tangent vector $\tau_\gamma\eqdef\frac{\gamma'}{|\gamma'|}$,  and the normal vector $\nu_\gamma \eqdef {\rm R}(\tau_\gamma)$, where ${\rm R}$ denotes counterclockwise rotation of $\tfrac\pi2$. We define $\de s_\gamma\eqdef |\gamma'| \de x$ the arclength element and $\partial_s\eqdef |\gamma'|^{-1}\partial_x$ the arclength derivative. The curvature of $\gamma$ is the vector
\[
\boldsymbol{k}_\gamma\eqdef \partial_s^2 \gamma.
\]
We shall usually drop the subscript $\gamma$ when there is no risk of confusion.

\begin{defn}[Networks]\label{def:Networks}
Fix $N\in\mathbb{N}$ and let $i\in\{1,\ldots, N\}$, $E^i:=[0,1]\times\{i\}$,
$E:=\bigcup_{i=1}^N E^i$ and 
$V:=\bigcup_{i=1}^N \{0,1\}\times\{i\}$.

Let $\sim$ be an equivalence relation that identifies points of $V$.
Let $G\eqdef E\,/\,\sim$ be the graph given by the topological quotient space of $E$ induced by $\sim$. Assume that $G$ is connected.

Denoting by $\pi:E\to G$ the projection onto the quotient, an \emph{endpoint} is a point $p \in G$ such that $\pi^{-1}(p) \subset V$ and it is a singleton, a \emph{junction} is a point $m \in G$ such that $\pi^{-1}(m) \subset V$ and it is not a singleton.

For any $(e^i,i) \in  E^i$ such that $\pi(e^i,i)$ is a junction, we assume that there exist exactly two boundary points $(e^j,j) \in E^j, (e^k,k) \in E^k$, for $j\neq k$ distinct from $i$, such that $\pi(e^i,i)=\pi(e^j,j)=\pi(e^k,k)$. A graph with such property is said to be \emph{regular}.

We denote by $J_G$ and $P_G$ the set of junctions and endpoints, respectively.

A \emph{network} is a continuous map $\Gamma: G\to \mathbb{R}^2$ such that $\Gamma|_{E^i}$ is either a $C^1$-embedding or a constant map. We say that $\Gamma$ is of class $C^k$ or smooth, if any embedding $\Gamma|_{E^i}$ is of class $C^k$ or $C^\infty$, respectively.

The \emph{inner tangent vector} of a regular curve $\gamma^i\eqdef\Gamma|_{E^i}$ at $e\in\{0,1\}$ is the vector
\[
(-1)^{e} \frac{(\gamma^i)'(e)}{|(\gamma^i)'(e)|}.
\]

A \emph{triple junctions} (or \emph{nondegenerate}) \emph{network} is an embedding $\Gamma:G\to \R^2$ such that any restriction $\Gamma|_{E^i}$ is nonconstant.

A network is said to be \emph{regular} if it is triple junctions and whenever $\pi(e^i,i)=\pi(e^j,j)=\pi(e^k,k)$ is a junction then any two inner tangent vectors of $\gamma^i,\gamma^j,\gamma^k$ at $e^i,e^j,e^k$, respectively, form an angle equal to $\tfrac23 \pi$.
\end{defn}

Without loss of generality, if $\Gamma:G\to\R^2$ is a network and $p \in G$ is an endpoint with $\pi^{-1}(p)=(e,i)$,
we will assume that $e=1$.

\bigskip

In order to treat blowups of the motion by curvature, we introduce the following definition of complete network.

\begin{defn}[Complete regular network]
Fix $N\in\mathbb{N}$ and let $i\in\{1,\ldots, N\}$. Let $E^i\eqdef I^i\times \{i\}$, where $I^i$ is any of the intervals $[0,1]$, $(0,1]$, $[1,0)$, or $(0,1)$. Let $E\eqdef \cup_i E^i$.

Let $\sim$ be an equivalence relation that identifies points on the boundaries of the $E^i$'s. Consider a graph $H\eqdef \cup_i E^i \,/\,\sim $ and let $\pi$ be the projection onto the quotient. Assume that $H$ is connected and that for any $(e^i,i) \in  E^i$ there exist exactly two boundary points $(e^j,j) \in E^j, (e^k,k) \in E^k$, for $j\neq k$ distinct from $i$, such that $\pi(e^i,i)=\pi(e^j,j)=\pi(e^k,k)$. Such a graph is said to be \emph{open}.

A continuous map $\Gamma: H\to \mathbb{R}^2$ such that $\Gamma|_{E^i}$ is either a $C^1$-embedding or a constant map is still said to be a \emph{network}.

A \emph{complete regular network} is an embedding $\Gamma:H\to \R^2$ such that $\Gamma|_{E^i}$ is a complete embedding of class $C^1$ for any $i$ and whenever $\pi(e^i,i)=\pi(e^j,j)=\pi(e^k,k)$ is a junction then any two inner tangent vectors of $\gamma^i,\gamma^j,\gamma^k$ at $e^i,e^j,e^k$, respectively, form an angle equal to $\tfrac23 \pi$.
\end{defn}

Whenever a graph $G$ is not explicitly said to be open, then it is assumed to be as in \cref{def:Networks}; in particular, such a $G$ is compact.

Observe that, by definition, a complete regular network has no endpoints. Also, a regular network $\Gamma:G\to \R^2$ such that $G$ has no endpoints is automatically a complete regular network. This choice is coherent with the fact that we shall restrict ourselves to the study of compact blowups.

\subsection{Motion by curvature and blowups}\label{sec:Motion}

\begin{defn}[Motion by curvature]\label{def:Motion}
A smooth one-parameter family of regular networks $\Gamma_t:G\to \R^2$, for $t \in [0,T)$, is a \textit{motion by curvature} if the parametrizations $\gamma^i_t$ of $\Gamma_t$ satisfy
\begin{equation}\label{eq:motion}
    \begin{split}
        \gamma^i_t(p)=P^i & \qquad \forall\, t \in [0,T), \,\,\forall\,p \text{ endpoint of $G$},\\
        \left\langle\partial_t\gamma^i(t,x),\nu^i(t,x)\right\rangle\nu^i(t,x)
        =\boldsymbol{k}^i(t,x)
        & \qquad \forall\, t \in [0,T), \,\,\forall\,x \in (0,1),
    \end{split}
\end{equation}
for some points $P^i \in \R^2$.
The solution is assumed to be maximal, i.e., we ask that it does not exists
another solution defined on $[0,\widetilde{T})$ with $\widetilde{T}>T$.
\end{defn}

For ease of notation, we shall often write that $\scal{\partial_t \Gamma_t, \nu_t}\nu_t = \boldsymbol{k}_t$ understanding that the relation holds along any edge $E^i$.

Exploiting the analysis carried out in \cite{GMP}, we can assume that the parametrizations of a motion by curvature evolve according to the \emph{special flow}, i.e., the evolution is determined by the non-degenerate parabolic second order equation
\begin{equation}\label{eq:SpecialFlow}
\partial_t \gamma^i_t = 
\frac{\partial^2_x\gamma^i_t}{|\partial_x\gamma^i_t|^2}\,.
\end{equation}

Throughout the paper, we will assume that $T<+\infty$ is the maximal time of existence for a motion by curvature, which thus develops a singularity at time $T$. In view of \cite{mannovplusch}, this means that either the inferior limit of the length of at least one curve of the network is zero as $t\to T^-$, or the superior limit of the $L^2$-norm of the curvature is $+\infty$ as $t\to T^-$.

\begin{defn}[Rescaled motion by curvature]\label{def:RescaledMotion}
If $\Gamma_t:G\to\R^2$ is a motion by curvature, the \emph{rescaled motion by curvature at $x_0\in\R^2$ and time $T$} is the flow
\begin{equation*}
    \widetilde{\Gamma}_{x_0}: [-\tfrac12\log T, +\infty)\times G \to \R^2
    \qquad\qquad
    \widetilde{\Gamma}_{x_0}(\t,x) \eqdef \frac{\Gamma_t(x)-x_0}{\sqrt{2(T-t)}},
\end{equation*}
where $\t(t)\eqdef -\tfrac12\log(T-t)$.
We shall denote $\widetilde{\Gamma}_\t(\cdot) \eqdef \widetilde{\Gamma}_{x_0}(\t,\cdot)$, omitting the dependence on $x_0$.
\end{defn}

Without loss of generality, in the definition of rescaled motion we shall always assume that $x_0=0$ is the origin of $\R^2$.

A direct computation shows that a rescaled motion by curvature evolves according to
\[
\partial_\t \widetilde{\Gamma}_\t = \sqrt{2(T-t)}\partial_t \Gamma_t + \widetilde{\Gamma}_\t,
\]
for $t(\t)=T-e^{-2\t}$, understanding that the relation holds along any edge $E^i$. In particular normal and tangent vectors of $\widetilde{\Gamma}_\t$ coincide with the ones of $\Gamma_t$, that is $\widetilde{\nu}_\t(x)=\nu_t(x)$ and $\widetilde{\tau}_\t(x)=\tau_t(x)$, and the normal velocity is
\[
\scal{\partial_\t \widetilde{\Gamma}_\t, \widetilde{\nu}_\t} \widetilde{\nu}_\t = \widetilde{\boldsymbol{k}}_\t + \scal{ \widetilde{\Gamma}_\t, \widetilde{\nu}_\t} \widetilde{\nu}_\t.
\]

In view of \cite[Section 7.2]{mannovplusch} and \cref{prop:FirstVariation} below, it turns out that a rescaled motion by curvature $\widetilde{\Gamma}_\t$ is an $L^2$-gradient flow of the Gaussian energy $\G$. Moreover, a complete regular network $\Gamma_*:H\to \R^2$ is critical point for $\G$, i.e., $\boldsymbol{k}_* + \scal{\Gamma_*,\nu_*}\nu_*=0$, if and only if $\Gamma_*$ is a \emph{shrinker}, that is, $\Gamma_*$ is complete regular and the motion by curvature starting from $\Gamma_*$ evolves homothetically contracting.

\begin{teo}[{\cite[Proposition 8.20, Remark 8.21]{mannovplusch}}]\label{thm:SubconvergenceRescaled}
Let $\widetilde{\Gamma}_\t$ be a rescaled motion by curvature at the origin $0$ and at maximal time $T$ of existence of a motion by curvature $\Gamma_t$. Assume that $0$ is not an endpoint of the motion $\Gamma_t$. Then the following hold.
\begin{enumerate}
    \item There exists a diverging sequence of times $\t_n\to+\infty$ and a (possibly empty) open subgraph $H\subset G$ such that: for any $E^i\subset G\setminus H$ and any $R>0$ it holds $\widetilde{\gamma}^i_{\t_n}(E^i)\cap B_R(0)=\emptyset$ for any $n$ large enough, for any $E^i\cap H \subset  H$ the curves $\widetilde{\gamma}^i_{\t_n}$ converge in $C^1 \cap H^2$ locally on $E^i\cap H$, up to reparametrization, to either a constant curve or to an embedding of class $H^2$.
    
    \item If also the sequence $\widetilde{\Gamma}_{\t_n}|_H$ converges to a complete regular network $\Gamma_*:H\to\R^2$ in $C^1 \cap H^2$ locally on $H$, up to reparametrizations, then convergence holds in $C^k$ locally on $H$ for any $k \in \N$.
\end{enumerate}
\end{teo}

In the notation of the first item in \cref{thm:SubconvergenceRescaled}, the network $\Gamma_*:H\to\R^2$ given by the limit immersions $\Gamma_*|_{E^i}\eqdef\lim_n \widetilde\gamma^i_{\t_n}$ is said to be a \emph{blowup of $\Gamma_t$ (at $0$ and time $T$)}.

If the second item in \cref{thm:SubconvergenceRescaled} occurs, we say that $\Gamma_*$ is a \emph{nondegenerate blowup of $\Gamma_t$ (at $0$ and time $T$)}.

\medskip

In the second item of \cref{thm:SubconvergenceRescaled}, the requirement that $\widetilde{\Gamma}_{\t_n}|_H$ converges to a complete regular network $\Gamma_*:H\to\R^2$ is equivalent to the fact that $\widetilde{\gamma}^i_{\t_n}$ converges to an embedding $\gamma^i_*$ for any $E^i\cap H \subset  H$ up to reparametrizations, and $\gamma^i_*(E^i\cap H)$ intersects $\gamma^j_*(E^j \cap H)$ for $i\neq j$ only at junctions, i.e., the limit network has multiplicity one.\\
We remark that nondegenerate blowups of a motion by curvature are automatically shrinkers.

An alternative and equivalent way for defining blowups is given by the already mentioned concept of tangent flow, recalled in the next remark.

\begin{rem}[Parabolic rescaling and tangent flows]\label{rem:EquivalenzaTangentRescaled}
Let $\Gamma_t:G\to\R^2$ be a motion by curvature defined on the time interval $[0,T)$. For any $\mu>0$, the \emph{parabolically rescaled motion by curvature $\Gamma^\mu_t$ (at point $x_0$ and time $T$)} is the evolution $\Gamma^\mu_t:G\to \R^2$ given by
\begin{equation*}
    \Gamma^\mu_t|_{E^i}(x) \eqdef \mu \big( \Gamma_{\mu^{-2}\,t+ T}|_{E^i}(x)- x_0\big),
\end{equation*}
for $t \in [-\mu^2T,0)$ and any $i=1,\ldots,N$ and $x \in E^i$. It is readily checked that $\Gamma^\mu_t$ is still a motion by curvature.\\
Suppose without loss of generality that $x_0=0$, then $\Gamma^\mu_t$ is related to the rescaled motion by curvature $\widetilde\Gamma_\t$ at time $T$ via the identities
\begin{equation*}
    \Gamma^\mu_t = \sqrt{-2t}\, \widetilde\Gamma_{\log(\mu/\sqrt{-t})},
    \qquad
    \widetilde\Gamma_\t = \frac{e^\t}{\mu\sqrt{2}}\Gamma^\mu_{-\mu^2e^{-2\t}}.
\end{equation*}
In particular $\Gamma^\mu_{-\frac12} = \widetilde\Gamma_{\log(\mu\sqrt{2})}$, whenever $-\mu^2T\le -\tfrac12$.\\
Consider now a sequence $\mu_j\nearrow+\infty$, together with the sequence of flows $\Gamma^{\mu_j}_t$. Exploiting the above identities and \cref{thm:SubconvergenceRescaled}, or (equivalently) \cite[Proposition 8.16, Lemma 8.18]{mannovplusch}, one can prove that, up to extracting a subsequence, there is an open subgraph $H\subset G$ such that: for any $E^i\subset G\setminus H$ and any $R>0$ it holds $\Gamma^{\mu_j}_t(E^i)\cap B_R(0)=\emptyset$ for any $t<0$ and $j$ large enough, for any $E^i\cap H \subset  H$ the curves $\Gamma^{\mu_j}_t|_{E^i}$ converge in $C^1 \cap H^2$ locally on $E^i\cap H$ to curves $\Gamma^\infty_t|_{E^i}$ for any $t<0$ as $j\to+\infty$, up to reparametrization, which are either constant or embeddings of class $H^2$. The flow $\Gamma^\infty_t$ is said to be a \emph{tangent flow} to $\Gamma_t$.\\
In case $\Gamma^\infty_t:H\to \R^2$ is regular for some (hence, any) $t$, then $\Gamma^\infty_t$ is a homothetically shrinking motion by curvature and $\Gamma^\infty_{-\frac12}$ is the limit of the rescaled motion $\widetilde\Gamma_\t$ along the sequence of times $\t_j=\log(\mu_j\sqrt{2})$.\\ Conversely, if $\Gamma_*$ is a sequential limit of $\widetilde\Gamma_\t$ along some sequence of times $\t_n$ as in \cref{thm:SubconvergenceRescaled}, then there exists the tangent flow $\Gamma^\infty_t$ obtained as the limit of the parabolic rescalings $\Gamma^{\mu_n}_t$ taking $\mu_n= e^{\t_n}/\sqrt{2}$.\\
Hence blowups of a motion by curvature are equivalently identified by tangent flows or by sequential limits of the rescaled motion by curvature, see also \cite[Remark 8.27, Remark 7.4]{mannovplusch}.
\end{rem}

\begin{rem}
In the notation and hypotheses of \cref{thm:SubconvergenceRescaled}, if item 2 occurs, if also $\Gamma_*(H)$ is compact and nonempty, then $H=G$ is a regular graph without endpoints and $\Gamma_*:G\to \R^2$ is a regular network.
\end{rem}

\section{Uniqueness of blowups}\label{sec:Blowups}

We need to recall the setup introduced in \cite{PludaPozzMinimalNetworks} in order to canonically parametrize networks close in $H^2$ to a regular one $\Gamma_*$ suitably as ``graphs'' over $\Gamma_*$.

For a given regular network $\Gamma_*$, for parametrizations of the form $\gamma^\ell(x) \eqdef \gamma^\ell_*(x) + \NN^\ell(x)\nu^\ell_*(x) + \TT^\ell(x)\tau^\ell_*(x)$, the next lemma contains sufficient conditions to be imposed on the functions $\NN^\ell,\TT^\ell$ so that the resulting parametrizations $\gamma^\ell$ define a triple junctions network. It readily seen that such conditions are also necessary \cite[Lemma 3.1]{PludaPozzMinimalNetworks}.

\begin{lemma}[{\cite[Lemma 3.2]{PludaPozzMinimalNetworks}}]\label{lem:SufficientConditionsNetwork}
Let $\Gamma_*:G\to \R^2$ be a regular network. 
Then there exists $\eps_{\Gamma_*}>0$ such that 
for every $\NN^\ell,\TT^\ell \in C^1([0,1])$, 
with $\|\NN^\ell\|_{C^1}, \|\TT^\ell\|_{C^1} \le \eps_{\Gamma_*}$ fulfilling at any junction $m=\pi(e^i,i)=\pi(e^j,j)=\pi(e^k,k)$
the following identities  
\begin{equation}\label{eq:ww}
\begin{split}
&\TT^i(e^i) = 
-\frac{1}{\sqrt{3}}\NN^i(e^i)
-\frac{2}{\sqrt{3}}(-1)^{e_i+e_j}\NN^j(e^j) \,,\\
&\TT^j(e^j) = 
\frac{2}{\sqrt{3}}(-1)^{e_i+e_j}\NN^i(e^i)   +\frac{1}{\sqrt{3}}\NN^j(e^j)\,,\\
&\TT^k(e^k) = 
-\frac{1}{\sqrt{3}}(-1)^{e_i+e_k}\NN^i(e^i)
+\frac{1}{\sqrt{3}}(-1)^{e_j+e_k}\NN^j(e^j)\,,\\
&(-1)^{e^i}\NN^i(e^i) + (-1)^{e^j}\NN^j(e^j) + (-1)^{e^k}\NN^k(e^k) = 0\,,
\end{split}
\end{equation}
the maps
\begin{equation*}
    \gamma^\ell(x) \eqdef \gamma^\ell_*(x) + \NN^\ell(x)\nu^\ell_*(x) + \TT^\ell(x)\tau^\ell_*(x),
\end{equation*}
define a triple junctions network.
\end{lemma}

The above result motivates the following definitions.

\begin{defn}
Let $G$ be a regular graph. At any junction $m=\pi(e^i,i)=\pi(e^j,j)=\pi(e^k,k)$, if $i<j<k$, we denote by $L^i,L^j,L^k$ the linear maps
\begin{equation*}
\begin{split}
    &L^i(a,b) = -\frac{1}{\sqrt{3}}a
    -\frac{2}{\sqrt{3}}(-1)^{e_i+e_j}b\,,\\
     &L^j(a,b) = \frac{2}{\sqrt{3}}(-1)^{e_i+e_j}a + 
     \frac{1}{\sqrt{3}}b\,,\\
     &L^k(a,b) =-\frac{1}{\sqrt{3}}(-1)^{e_i+e_k}a + \frac{1}{\sqrt{3}}(-1)^{e_j+e_k}b\,.
\end{split}
\end{equation*}
Moreover, we denote by $I_m$ the set of indices $\ell$ such that $E^\ell$ has an endpoint at $m$, and we denote by $e^\ell_m\in\{0,1\}$ the endpoint of $E^\ell$ at $m$, for $\ell \in I_m$.\\
Furthermore, for $\ell \in I_m$, we denote by $\Ell^\ell_m$ the linear operator $L^i,L^j,L^k$, depending on whether $\ell$ is the minimal, intermediate, or maximal index $\ell \in I_m$.

\end{defn}

From now on and for the rest of the paper, we fix a nonincreasing smooth cut-off function
\begin{equation*}
    \chi:[0,1/2]\to [0,1], 
    \quad 
    \chi|_{[0,\frac18]}=1,
    \quad
    \chi|_{[\frac38,\frac12]}\equiv 0.
\end{equation*}

\begin{defn}[Adapted tangent functions]\label{def:Adapted}
Let $G$ be a regular graph. Let $\NN^i,\TT^i:[0,1]\to \R$ be functions of class $C^1$, for $i=1,\ldots,N$. We say that the $\TT^i$'s are \emph{adapted} to the $\NN^i$'s whenever
\begin{equation}\label{eq:DefAdattati}
    \TT^\ell(|e^\ell_m - x|) = \chi(x) \Ell^\ell_m(\NN^i(|e^i_m - x|), \NN^j(|e^j_m-x|) ),
\end{equation}
for $x \in [0,\tfrac12]$ for any junction $m=\pi(e^i,i)=\pi(e^j,j)=\pi(e^k,k)$, and
\begin{equation*}
    \TT^i(x)=0
\end{equation*}
for $x \in [\tfrac12,1]$ for any endpoint $\pi(1,i)$.
\end{defn}

The next proposition introduce a canonical way of writing a one-parameter family of triple junctions networks close in $H^2$ to a regular one $\Gamma_*$ as graphs over $\Gamma_*$.\\
We shall apply the next result also to the particular case in which $\Gamma_t\equiv \Gamma$ for any $t$.

\begin{prop}[{\cite[Proposition 3.4, Corollary 3.5]{PludaPozzMinimalNetworks}}]\label{prop:ParametrizzazioneTNtempo}
Let $\Gamma_*:G\to \R^2$ be a smooth regular network of a graph $G$ without endpoints. Then there exist $\eps_{\Gamma_*}>0$ such that whenever $\Gamma_t:G\to \R^2$ is a one-parameter family of triple junctions networks of class $H^2$, differentiable with respect to $t$ for $t \in [t_0-h,t_0+h]$ and $h>0$, such that $(t,x)\mapsto \partial_t \gamma^i_t(x)$ is continuous for any $i$ and
\begin{align*}
     \sum_i \| \gamma^i_* - \gamma^i_t \|_{H^2(\d x)}
     \le \eps_{\Gamma_*},
\end{align*}
for any $i,t$, then there exist $h'\in(0,h)$ and functions $\NN^i_t, \TT^i_t \in H^2(\d x)$ and reparametrizations $\varphi^i_t:[0,1]\to[0,1]$ of class $H^2(\d x)$, continuously differentiable with respect to $t$ for $t \in [t_0-h',t_0+h']$ such that
\begin{equation*}
      \gamma^i_t\circ \varphi^i_t(x) = \gamma^i_*(x) + \NN^i_t(x)\nu^i_*(x) + \TT^i_t(x)\tau^i_*(x).
\end{equation*}

Moreover, the $\TT^i_t$'s are adapted to the $\NN^i_t$'s and
\begin{itemize}
    \item for any $\delta>0$ there is $\eps \in (0,\eps_{\Gamma_*})$ such that
    \begin{equation}\label{eq:StimaH2NiTempo}
        \sum_i \| \gamma^i_* - \gamma^i_t \|_{H^2(\d x)} \le \eps \quad\forall\,t
    \quad\implies\quad  
    \sum_i \|\NN^i_t\|_{H^2(\d x)}+ \|\varphi^i_t(x)-x\|_{H^2(\d x)} \le \delta,
    \end{equation}
    for any $t \in [t_0-h',t_0+h']$;
    
    \item for any $\eta>0$ and $m\in \N$ there is $\eps_{\eta,m}\in(0,\eps_{\Gamma_*})$ and $h_{\eta,m} \in (0,h)$ such that if $\sum_i \| \gamma^i_* - \gamma^i_t \|_{C^{m+1}([0,1])} \le \eps$ for any $t$, then
    \begin{equation}\label{eq:StimaSobolevNiTempo}
    \sum_i \| \NN^i_t\|_{H^m(\d x)} \le \eta,
    \end{equation}
    for any $t \in [t_0-h_{\eta,m}, t_0 + h_{\eta,m}]$.
    \end{itemize}
\end{prop}

\subsection{\L ojasiewicz--Simon inequality for $\G$}

In this section we prove the key \L ojasiewicz--Simon inequality for the functional $\G$. We need to compute first and second variations of $\G$ taken with respect to variations on the normal functions $\NN^i$ (and, consequently, on the adapted tangent functions).

\begin{prop}\label{prop:FirstVariation}
Let $\Gamma_*:G\to \R^2$ be a smooth regular network where $G$ is a graph without endpoints. Then there is $\eps_{\Gamma_*}>0$ such that the following holds.

Let $\NN^i, X^i \in H^2$ with $\|\NN^i\|_{H^2}\le \eps_{\Gamma_*}$ such that
\begin{equation*}
\sum_{\ell \in I_m} (-1)^{e^\ell_m} \NN^\ell(e^\ell_m) = \sum_{\ell \in I_m} (-1)^{e^\ell_m} X^\ell(e^\ell_m)=0
\qquad
\forall m \in J_G.
\end{equation*}
Let $\Gamma^\eps:G\to \R^2$ be the triple junctions network defined by
\begin{equation}\label{eq:zzz}
\gamma^{i,\eps}(x) \eqdef \gamma^i_*(x) + (\NN^i(x)+ \eps X^i(x))\nu^i_*(x) + \TT^{i,\eps}(x)\tau^i_*(x),
\end{equation}
for any $i$, for any $|\eps|<\eps_0$ and some $\eps_0>0$, where the $\TT^{i,\eps}$'s are adapted to the $(\NN^i+ \eps X^i)$'s, for any $|\eps|<\eps_0$.

Call $\Gamma$ the network given by the immersions $\gamma^i\eqdef\gamma^{i,0}$.
Then
\begin{equation}\label{eq:FirstVariation}
\begin{split}
    \frac{\d}{\d \eps}  \G (\Gamma^\eps) \bigg|_0 
    &= 
    \sum_{m \in J_G}  \sum_{\ell \in I_m} 
    (-1)^{1+e^\ell_m} \left[ 
    \scal{\tau^\ell(e^\ell_m), \nu^\ell_*(e^\ell_m)} + \sum_{j \in I_m} h_{\ell j}\scal{\tau^j(e^j_m) , \tau^j_*(e^j_m)}
     \right] e^{-|\gamma^{\ell}(e^\ell_m)|^2/2}
     X^\ell(e^\ell_m) + \\
    &\phantom{=} - \sum_i 
    \int_0^1 \bigg( 
    \scal{(\gamma^i)^\perp+\boldsymbol{k}^i,\nu^i_*}e^{-|\gamma^i|^2/2}|\partial_x \gamma^i| 
    +\\
    &\qquad\qquad\qquad
   + \sum_j f_{ij}\chi \scal{(\gamma^j)^\perp+\boldsymbol{k}^j,\tau^j_*}e^{-|\gamma^j|^2/2}|\partial_x \gamma^j| +\\
    &\qquad\qquad\qquad
    + g_{ij}\chi(1-x) \scal{(\gamma^j)^\perp+\boldsymbol{k}^j,\tau^j_*}(1-x)e^{-|\gamma^j(1-x)|^2/2}|\partial_x \gamma^j|(1-x) 
    \bigg) X^i \de x,
\end{split}
\end{equation}
where $h_{\ell,j}, f_{ij}, g_{ij}\in \R$ only depend on the topology of $G$.

If also $\Gamma$ is regular, then
\begin{equation}\label{eq:FirstVariationRegular}
    \begin{split}
        \frac{\d}{\d \eps}  \G (\Gamma^\eps) \bigg|_0 
    &= 
    - \sum_i 
    \int_0^1 \bigg( 
    \scal{(\gamma^i)^\perp+\boldsymbol{k}^i,\nu^i_*}e^{-|\gamma^i|^2/2}|\partial_x \gamma^i| 
    +\\
    &\qquad\qquad\qquad
   + \sum_j f_{ij}\chi \scal{(\gamma^j)^\perp+\boldsymbol{k}^j,\tau^j_*}e^{-|\gamma^j|^2/2}|\partial_x \gamma^j| +\\
    &\qquad\qquad\qquad
    + g_{ij}\chi(1-x) \scal{(\gamma^j)^\perp+\boldsymbol{k}^j,\tau^j_*}(1-x)e^{-|\gamma^j(1-x)|^2/2}|\partial_x \gamma^j|(1-x) 
    \bigg) X^i \de x.
    \end{split}
\end{equation}
\end{prop}

\begin{proof}
Let us assume first that there is a junction $m$ such that the functions $X^\ell$ appearing in~\eqref{eq:zzz} all vanish except for $\ell \in I_m$. Moreover, for $\ell \in I_m$, assume that $X^\ell$ has compact support in $[0,\tfrac58)$.

Let us denote $m=\pi(e^i,i)=\pi(e^j,j)=\pi(e^k,k)$, where $i<j<k$. Differentiating we get
\begin{equation}\label{eq:www}
    \begin{split}
        \frac{\d}{\d \eps} \G (\Gamma^\eps)
        =  \sum_{\ell \in I_m}\int_0^1 \left(
        -\scal{\gamma^{\ell,\eps}, \partial_\eps\gamma^{\ell,\eps}}+
        \frac{1}{|\partial_x \gamma^{\ell,\eps}|^2} \scal{\partial_x \gamma^{\ell,\eps}, \partial_x \partial_\eps\gamma^{\ell,\eps}}\right) e^{-|\gamma^{\ell,\eps}|^2/2} |\partial_x \gamma^{\ell,\eps}| \de x,
    \end{split}
\end{equation}
indeed, since ${\rm spt}\,X^\ell \subset [0,\tfrac58)$, by \eqref{eq:DefAdattati} and definition of $\chi$, we have that $\TT^{n,\eps}$ does not depend on $\eps$ for all $n \not \in I_m$. More explicitly
\begin{equation}\label{eq:zxzxzx}
    \begin{split}
    \TT^{\ell,\eps}(x) &= \chi(x) \Ell^\ell_m((\NN^i+ \eps X^i)(x), (\NN^j+\eps X^j)(x) ) \\
    &=  \chi(x) \Ell^\ell_m(\NN^i(x), \NN^j(x) ) + \eps \, \chi(x)\Ell^\ell_m(X^i(x), X^j(x) )
\end{split}
\end{equation}
for $\ell \in I_m$, hence, letting
\begin{equation}\label{eq:zww}
    Y^\ell \eqdef \partial_\eps\gamma^{\ell,\eps} = X^\ell\nu^\ell_* + \chi \Ell^\ell_m(X^i,X^j)\tau^\ell_*,
\end{equation}
and integrating by parts, we find
\begin{equation*}
    \begin{split}
        \frac{\d}{\d \eps} \G (\Gamma^\eps) \bigg|_0
        &= \sum_{\ell \in I_m}\int_0^1
        -\scal{\gamma^{\ell}, Y^\ell} e^{-|\gamma^{\ell}|^2/2} |\partial_x \gamma^{\ell}|
        -\left\langle \partial_x\left(\frac{\partial_x\gamma^\ell}{|\partial_x\gamma^\ell|} e^{-|\gamma^{\ell}|^2/2}\right) , Y^\ell\right\rangle \de x + \\
        &\qquad+  \sum_{\ell \in I_m} (-1)^{e^\ell_m+1} \left\langle \frac{\partial_x\gamma^\ell(e^\ell_m)}{|\partial_x\gamma^\ell(e^\ell_m)|} e^{-|\gamma^{\ell}(e^\ell_m)|^2/2}, Y^\ell(e^\ell_m)\right\rangle.
    \end{split}
\end{equation*}
Since $\gamma^{\ell,\eps}(e^\ell_m)=\gamma^{l,\eps}(e^l_m)$ for any $\eps$ and $\ell,l \in I_m$ by \cref{lem:SufficientConditionsNetwork}, then $Y^\ell(e^\ell_m)=Y^l(e^l_m)$ for any $\ell,l \in I_m$. Hence, if $\Gamma$ is also regular, we observe that the boundary term above vanishes, that is $\sum_{\ell \in I_m} (-1)^{e^\ell_m+1} \left\langle \frac{\partial_x\gamma^\ell}{|\partial_x\gamma^\ell|} e^{-|\gamma^{\ell}|^2/2}, Y^\ell\right\rangle \big|_{e^\ell_m}=0$.

Moreover
\begin{equation*}
    \begin{split}
        \frac{\d}{\d \eps} \G (\Gamma^\eps) \bigg|_0
        &=
        \sum_{\ell \in I_m} (-1)^{e^\ell_m+1} \left\langle \tau^\ell(e^\ell_m), Y^\ell(e^\ell_m)\right\rangle e^{-|\gamma^{\ell}(e^\ell_m)|^2/2} + \\
        &\qquad - \sum_{\ell \in I_m} \int_{\gamma^\ell} \scal{\gamma^{\ell}, Y^\ell} e^{-|\gamma^{\ell}|^2/2}
        + \scal{\boldsymbol{k}^\ell, Y^\ell} e^{-|\gamma^{\ell}|^2/2}
        -\scal{\gamma^\ell, \tau^\ell}\scal{\tau^\ell, , Y^\ell} e^{-|\gamma^{\ell}|^2/2} \\
        &= \sum_{\ell \in I_m} (-1)^{e^\ell_m+1} \left\langle \tau^\ell(e^\ell_m), Y^\ell(e^\ell_m)\right\rangle e^{-|\gamma^{\ell}(e^\ell_m)|^2/2} -  \sum_{\ell \in I_m} \int_{\gamma^\ell} \scal{(\gamma^{\ell})^\perp +\boldsymbol{k}^\ell, Y^\ell}  e^{-|\gamma^{\ell}|^2/2}.
    \end{split}
\end{equation*}
Inserting \eqref{eq:zww} then \eqref{eq:FirstVariation} follows, and \eqref{eq:FirstVariationRegular} follows as well as we observed that the boundary terms vanish when $\Gamma$ is regular.
\end{proof}

\begin{prop}\label{prop:SecondVariation}
Let $\Gamma_*:G\to \R^2$ be a smooth regular network where $G$ is a graph without endpoints.
Let $ X^i \in H^2$ such that
\begin{equation*}
\sum_{\ell \in I_m} (-1)^{e^\ell_m} X^\ell(e^\ell_m)=0
\qquad
\forall m \in J_G.
\end{equation*}
Let $\Gamma^\eps:G\to \R^2$ be the triple junctions network defined by
\begin{equation}\label{eq:zzz2}
\gamma^{i,\eps}(x) \eqdef \gamma^i_*(x) + \eps X^i(x)\nu^i_*(x) + \TT^{i,\eps}(x)\tau^i_*(x),
\end{equation}
for any $i$, for any $|\eps|<\eps_0$ and some $\eps_0>0$, where the $\TT^{i,\eps}$'s are adapted to the $(\eps X^i)$'s, for any $|\eps|<\eps_0$

Then
\begin{equation}\label{eq:SecondVariation}
    \begin{split}
        \frac{\d^2}{\d\eps^2} \G (\Gamma^\eps)\bigg|_0 
        &= 
        \sum_{m \in J_G} \sum_{\ell \in I_m} (-1)^{1+e^\ell_m} \left(\partial_s X^\ell (e^\ell_m)  + \sum_{l \in I_m}\omega^l_m(X^l(e^l_m))
        \right) X^\ell (e^\ell_m) e^{-|\gamma^\ell_*(e^\ell_m)|^2/2} + \\
        &\phantom{=}- \sum_i \int_0^1 \bigg( \partial^2_s X^i 
        +X^i
        + \sum_j \Omega_{ij}(\overline{X}) \bigg) X^i e^{-|\gamma^i_*|^2/2}\de s_{\gamma^i_*},
    \end{split}
\end{equation}
where $\partial_s X^n= |\partial_x\gamma^n_*|^{-1}\partial_x X^n$ for any $n$, $\overline{X}\eqdef (X^1,\ldots, X^N)$, $\omega^l_m:\R\to\R$ is linear, and $\Omega_{ij}: [H^2(0,1)]^N\to L^2(0,1)$ is a compact linear operator.
\end{prop}

\begin{proof}
As in the proof of \cref{prop:FirstVariation}, let us assume first that there is a junction $m$ such that the functions $X^\ell$ appearing in \eqref{eq:zzz2} all vanish except for $\ell \in I_m$. Moreover, for $\ell \in I_m$, assume that $X^\ell$ has compact support in $[0,\tfrac58)$.

Let us denote $m=\pi(e^i,i)=\pi(e^j,j)=\pi(e^k,k)$, where $i<j<k$. Recalling \eqref{eq:www} and setting
\begin{equation}\label{eq:zz4}
    Y^\ell \eqdef \partial_\eps\gamma^{\ell,\eps} = X^\ell\nu^\ell_* + \chi \Ell^\ell_m(X^i,X^j)\tau^\ell_*,
\end{equation}
we get
\begin{equation}\label{eq:zz6}
    \begin{split}
        \frac{\d^2}{\d \eps^2} \G (\Gamma^\eps)\bigg|_0
        &=\sum_{\ell \in I_m}
        \int_0^1 \left(-|Y^\ell|^2 - \frac{2}{|\partial_x\gamma^\ell_*|^4} \scal{\partial_x\gamma^\ell_*, \partial_x Y^\ell}^2 +\frac{1}{|\partial_x\gamma^\ell_*|^2}|\partial_x Y^\ell|^2
        \right) e^{-|\gamma^\ell_*|^2/2} |\partial_x\gamma^\ell_*| + \\
        &\qquad\qquad
        + \left(-\scal{\gamma^\ell_*,Y^\ell} + \frac{1}{|\partial_x\gamma^\ell_*|^2}\scal{\partial_x\gamma^\ell_*,\partial_x Y^\ell}
        \right)^2 e^{-|\gamma^\ell_*|^2/2} |\partial_x\gamma^\ell_*| 
        \de x  \\
        &= \sum_{\ell \in I_m} \int_0^1 
         \left(-|Y^\ell|^2 - 2 \scal{\tau^\ell_*, \partial_s Y^\ell}^2 +|\partial_s Y^\ell|^2
        \right) e^{-|\gamma^\ell_*|^2/2}  + \\
        &\qquad\qquad
        + \left(-\scal{\gamma^\ell_*,Y^\ell} + \scal{\tau^\ell_*,\partial_s Y^\ell}
        \right)^2 e^{-|\gamma^\ell_*|^2/2} 
        \de s  \\
        &= \sum_{\ell \in I_m}
        (-1)^{e^\ell_m+1}\scal{\partial_s Y^\ell(e^\ell_m),  Y^\ell(e^\ell_m)} e^{-|\gamma^\ell_*(e^\ell_m)|^2/2}
        + \\
        &\qquad+\sum_{\ell \in I_m} \int_0^1 \left(
        -\scal{\partial^2_s Y^\ell,Y^\ell} +  \scal{\partial_s Y^\ell,Y^\ell} \scal{\gamma^\ell_*,\tau^\ell_*}
        -|Y^\ell|^2 - 2 \scal{\tau^\ell_*, \partial_s Y^\ell}^2 \right) e^{-|\gamma^\ell_*|^2/2}  + \\
        &\qquad\qquad
        + \left(-\scal{\gamma^\ell_*,Y^\ell} + \scal{\tau^\ell_*,\partial_s Y^\ell}
        \right)^2 e^{-|\gamma^\ell_*|^2/2} 
        \de s \\
        &=  \sum_{\ell \in I_m}
        (-1)^{e^\ell_m+1}\scal{\partial_s Y^\ell(e^\ell_m),  Y^\ell(e^\ell_m)} e^{-|\gamma^\ell_*(e^\ell_m)|^2/2}
        + \\
        &\qquad+\sum_{\ell \in I_m} \int_0^1 \bigg(
        -\scal{\partial^2_s Y^\ell,Y^\ell} 
        -|Y^\ell|^2 
        -  \scal{\tau^\ell_*, \partial_s Y^\ell}^2  + \\
        &\qquad\qquad
        +  \scal{\partial_s Y^\ell,Y^\ell} \scal{\gamma^\ell_*,\tau^\ell_*}
        +\scal{\gamma^\ell_*,Y^\ell}^2 -2\scal{\gamma^\ell_*, Y^\ell} \scal{\tau^\ell_*,\partial_s Y^\ell}
        \bigg) e^{-|\gamma^\ell_*|^2/2} 
        \de s .
    \end{split}
\end{equation}
We integrate by parts
\begin{equation*}
    \begin{split}
        \sum_{\ell \in I_m} &
        \int_0^1 -  \scal{\tau^\ell_*, \partial_s Y^\ell}^2 e^{-|\gamma^\ell_*|^2/2} 
        \de s 
        = \\
        &=\sum_{\ell \in I_m}
        (-1)^{e^\ell_m} \scal{\tau^\ell_*, \partial_s Y^\ell}\scal{\tau^\ell_*,  Y^\ell} e^{-|\gamma^\ell_*|^2/2} \bigg|_{e^\ell_m}
        +\\
        &\qquad+ \sum_{\ell \in I_m} \int_0^1 \bigg(
        \scal{\tau^\ell_*, \partial_s^2 Y^\ell}\scal{\tau^\ell_*,  Y^\ell}
        +\scal{\boldsymbol{k}^\ell_*, \partial_s Y^\ell}\scal{\tau^\ell_*,  Y^\ell}
        + \scal{\tau^\ell_*, \partial_s Y^\ell}\scal{\boldsymbol{k}^\ell_*,  Y^\ell} + \\
        &\qquad\qquad
        - \scal{\tau^\ell_*, \partial_s Y^\ell}\scal{\tau^\ell_*,  Y^\ell}\scal{\gamma^\ell_*, \tau^\ell_*}
        \bigg)e^{-|\gamma^\ell_*|^2/2} 
        \de s .
    \end{split}
\end{equation*}
Inserting into \eqref{eq:zz6} we get
\begin{equation}\label{eq:zz7}
    \begin{split}
        \frac{\d^2}{\d \eps^2} \G (\Gamma^\eps)\bigg|_0
        &= \sum_{\ell \in I_m}
        (-1)^{e^\ell_m+1}\big( \scal{\partial_s Y^\ell(e^\ell_m),  Y^\ell(e^\ell_m)} - \scal{\tau^\ell_*(e^\ell_m), \partial_s Y^\ell (e^\ell_m)}\scal{\tau^\ell_*(e^\ell_m),  Y^\ell (e^\ell_m)} \big) e^{-|\gamma^\ell_*(e^\ell_m)|^2/2}
        + \\
        &\qquad+\sum_{\ell \in I_m}  \int_0^1
        \big\langle -\partial_s^2 Y^\ell + \scal{\partial_s^2 Y^\ell, \tau^\ell_*}\tau^\ell_* - Y^\ell +
        \Omega^\ell(Y^\ell)
        ,
        Y^\ell\big\rangle e^{-|\gamma^\ell_*|^2/2} 
        \de s ,
    \end{split}
\end{equation}
where $\Omega^\ell:H^2(0,1)\to L^2(0,1)$ is a compact operator. Inserting the definition of $Y^\ell$, we see that the second order term with respect to $Y^\ell$ in \eqref{eq:zz7} becomes
\begin{equation*}
    \begin{split}
        -\partial_s^2 Y^\ell + \scal{\partial_s^2 Y^\ell, \tau^\ell_*}\tau^\ell_*
        &= -\partial_s^2 X^\ell \nu^\ell_* - \partial^2_s (\chi \Ell^\ell_m(X^i,X^j) ) \tau^\ell_* +  \partial^2_s (\chi \Ell^\ell_m(X^i,X^j) ) \tau^\ell_*  + \sum_{j \in I_m} \widetilde\Omega_{\ell j}(X^j) \\
        &=  \partial_s^2 X^\ell \nu^\ell_* + \sum_{j \in I_m} \widetilde\Omega_{\ell j}(X^j),
    \end{split}
\end{equation*}
where $\widetilde\Omega_{\ell j}:H^2(0,1)\to [L^2(0,1)]^2$ is a compact operator.\\
Moreover the boundary terms in \eqref{eq:zz7} are
\begin{equation*}
    \begin{split}
        \scal{(\partial_sY^\ell)^\perp, Y^\ell}\big|_{e^\ell_m}
        &=\scal{(\partial_sY^\ell)^\perp, \nu^\ell_*}X^\ell\big|_{e^\ell_m}\\
        &
        =\partial_s X^\ell (e^\ell_m) \, X^\ell (e^\ell_m) + \Ell^\ell_m(X^i(e^i_m),X^j(e^j_m)) \scal{\boldsymbol{k}^\ell_*(e^\ell_m),\nu^\ell_*(e^\ell_m)} X^\ell (e^\ell_m).
    \end{split}
\end{equation*}
Hence \eqref{eq:zz7} eventually implies \eqref{eq:SecondVariation}.
\end{proof}

\bigskip

We now introduce the linear functional analytic setting needed for the proof of the \L ojasiewicz--Simon inequality.\\
For a fixed shrinker $\Gamma_*:G\to \R^2$, where $G$ is a graph without endpoints, we denote by $M\eqdef \sharp J_G$ the cardinality of junctions, and we define the space
\begin{equation}\label{eq:DefV}
\begin{split}
    V \eqdef 
     \bigg\{ \overline{\NN}\eqdef(\NN^1,\ldots,\NN^N) \in [H^2(0,1)]^N  \st 
    &\sum_{\ell \in I_m} (-1)^{e^\ell_m} \NN^\ell(e^\ell_m) =0 \,\, \forall\, m \in J_G 
    \bigg\},
\end{split}
\end{equation}
endowed with the complete scalar product
\begin{equation*}
    \scal{\overline{\NN},\overline{X}}_V \eqdef \sum_i \int_0^1 \big( \NN^iX^i + \partial_s^2\NN^i\partial_s^2X^i \big) e^{-|\gamma^i_*|^2/2} \de s^i_*,
\end{equation*}
where $\d s^i_*\eqdef \d s_{\gamma^i_*}$. We further define the $2$-dimensional spaces
\begin{equation*}
    W_m \eqdef \left\{(v^\ell_m)_{\ell \in I_m} \in \R^3 \st \sum_{\ell \in I_m}(-1)^{e^\ell_m} v^\ell_m=0 \right\},
\end{equation*}
for $m=1,\ldots,M$, where $W_m$ is endowed with the scalar product
\begin{equation*}
    \scal{(v^\ell_m),(w^\ell_m)}_{W_m}\eqdef \sum_{\ell \in I_m}v^\ell_m w^\ell_m e^{-|\gamma^\ell_*(e^\ell_m)|^2/2}.
\end{equation*}
Finally, let $Z$ be the product space
\begin{equation}\label{eq:DefZ}
    Z\eqdef \Pi_{m=1}^M W_m \times \Pi_{i=1}^N L^2(e^{-|\gamma^i_*|^2/2}\de s^i_*) \eqqcolon \Pi_{m=1}^M W_m \times \Pi_{i=1}^N L^2(\d\mu^i_*) .
\end{equation}
Observe that ${\rm j}:V\hookrightarrow Z$ compactly with the natural injection
\begin{equation}\label{eq:Injection}
\overline{\NN} \quad\overset{\rm j}{\mapsto}\quad \left( (\NN^\ell(e^\ell_m)) ,  \overline{\NN}\right)
\end{equation}

For $r_{\Gamma_*}>0$ small enough, we define the energy $\boldsymbol{{\rm G}}:B_{r_{\Gamma_*}} (0) \subset V\to [0,+\infty)$ by
\begin{equation}\label{eq:DefG}
    \boldsymbol{{\rm G}}(\overline{\NN}) \eqdef \sum_i \G \left(\gamma^i_* + \NN^i \nu^i_* + \TT^i \tau^i_* \right),
\end{equation}
where the $\TT^i$'s are adapted to the $\NN^i$'s (see \cref{def:Adapted}). Recalling \cref{lem:SufficientConditionsNetwork}, the immersions $\gamma^i_* + \NN^i \nu^i_* + \TT^i \tau^i_* $ define a triple junctions network.

In the next lemma we check the validity of sufficient conditions implying the desired inequality.

\begin{lemma}\label{lem:SuffConditionsLoja}
Let $\Gamma_*:G\to \R^2$ be a shrinker, where $G$ is a graph without endpoints. Let $V,Z,\boldsymbol{\rm G}$ be as above, and identify $Z^\star$ with ${\rm j}^\star(Z^\star)\subset V^\star$, for ${\rm j}$ as in \eqref{eq:Injection}.

Then the following hold.
\begin{enumerate}
    \item The functional $\boldsymbol{{\rm G}}:B_{r_{\Gamma_*}} (0) \subset V\to [0,+\infty)$ is analytic.

    \item The first variation operator $\delta \boldsymbol{\rm G} : V\to Z^\star$ is $Z^\star$-valued by setting
    \begin{equation}\label{eq:FirstVariationBold}
    \begin{split}
        \delta &\boldsymbol{\rm G} (\overline{\NN})[((v^\ell_m), \overline{X})]
    \\&= 
    \sum_{m \in J_G} \left\langle \bigg(
    (-1)^{1+e^\ell_m} \bigg[ 
    \scal{\tau^\ell(e^\ell_m), \nu^\ell_*(e^\ell_m)} + \sum_{j \in I_m} h_{\ell j}\scal{\tau^j(e^j_m) , \tau^j_*(e^j_m)}
     \bigg] e^{-|\gamma^{\ell}(e^\ell_m)|^2/2}e^{|\gamma^{\ell}_*(e^\ell_m)|^2/2} \bigg)
    ,(v^\ell_m)\right\rangle_{W_m} \\
    &\quad
    - \sum_i \bigg\langle
    \scal{(\gamma^i)^\perp+\boldsymbol{k}^i,\nu^i_*}e^{-|\gamma^i|^2/2}|\partial_x \gamma^i| e^{|\gamma^i_*|^2/2}|\partial_x \gamma^i_*|^{-1} 
    +\\
    &\qquad\qquad
   + \sum_j f_{ij}\chi \scal{(\gamma^j)^\perp+\boldsymbol{k}^j,\tau^j_*}e^{-|\gamma^j|^2/2}|\partial_x \gamma^j|e^{|\gamma^i_*|^2/2}|\partial_x \gamma^i_*|^{-1} +\\
    &\qquad\qquad
    + g_{ij}\chi(1-x) \scal{(\gamma^j)^\perp+\boldsymbol{k}^j,\tau^j_*}(1-x)e^{-|\gamma^j(1-x)|^2/2}|\partial_x \gamma^j|(1-x) e^{|\gamma^i_*|^2/2}|\partial_x \gamma^i_*|^{-1}
    , X^i
    \bigg\rangle_{L^2(\d\mu^i_*)}
\end{split}
    \end{equation}
where $f_{ij}, g_{ij}, h_{\ell j} \in \R$ depend on the topology of $G$, and $\tau^i, \boldsymbol{k}^i$ are referred to the immersions $\gamma^i \eqdef \gamma^i_*  + \NN^i \nu^i_* + \TT^i\tau^i_*$, with $\TT^i$ adapted to $\NN^i$.\\
Moreover $\delta \boldsymbol{\rm G} : V\to Z^\star$ is analytic.

If also the network defined by the immersions $\gamma^i$ is regular, then
\begin{equation}\label{eq:FirstVariationBoldRegular}
    \begin{split}
        \delta &\boldsymbol{\rm G}  (\overline{\NN})[((v^\ell_m), \overline{X})]
    \\&=
   - \sum_i \bigg\langle
    \scal{(\gamma^i)^\perp+\boldsymbol{k}^i,\nu^i_*}e^{-|\gamma^i|^2/2}|\partial_x \gamma^i| e^{|\gamma^i_*|^2/2}|\partial_x \gamma^i_*|^{-1} 
    +\\
    &\qquad\qquad
   + \sum_j f_{ij}\chi \scal{(\gamma^j)^\perp+\boldsymbol{k}^j,\tau^j_*}e^{-|\gamma^j|^2/2}|\partial_x \gamma^j|e^{|\gamma^i_*|^2/2}|\partial_x \gamma^i_*|^{-1} +\\
    &\qquad\qquad
    + g_{ij}\chi(1-x) \scal{(\gamma^j)^\perp+\boldsymbol{k}^j,\tau^j_*}(1-x)e^{-|\gamma^j(1-x)|^2/2}|\partial_x \gamma^j|(1-x) e^{|\gamma^i_*|^2/2}|\partial_x \gamma^i_*|^{-1}
    , X^i
    \bigg\rangle_{L^2(\d\mu^i_*)}.
    \end{split}
\end{equation}
    
\item The second variation $\delta^2 \boldsymbol{\rm G}_0 : V\to Z^\star $ at $0$ is $Z^\star$-valued by setting
    \begin{equation}\label{eq:SecondVariationBold}
    \begin{split}
        \delta^2 \boldsymbol{\rm G}_0 ( \overline{X} ) [((v^\ell_m), \overline{Z}) ]
        &=\sum_{m \in J_G}\left\langle \bigg((-1)^{1+e^\ell_m} \bigg(\partial_s X^\ell (e^\ell_m)  + \sum_{l \in I_m}\omega^l_m(X^l(e^l_m))
        \bigg) \bigg) 
        ,(v^\ell_m)\right\rangle_{W_m} + \\
        &\phantom{=}-
        \sum_i
        \left\langle   \partial^2_s X^i 
        +X^i
        + \sum_j \Omega_{ij}(\overline{X}) 
        ,Z^i \right\rangle_{L^2(\d\mu^i_*)} ,
    \end{split}
\end{equation}
where $\partial_s X^n= |\partial_x\gamma^n_*|^{-1}\partial_x X^n$ for any $n$, $\overline{X}\eqdef (X^1,\ldots, X^N)$, $\omega^l_m:\R\to\R$ is linear, and $\Omega_{ij}:V\to L^2(\d\mu^i_*)$ is a compact linear operator.\\
Moreover $\delta^2 \boldsymbol{\rm G}_0 : V\to Z^\star $ is a Fredholm operator of index zero.
\end{enumerate}
\end{lemma}

\begin{proof}
The proof follows by adapting the arguments in \cite[Section 3.3]{PludaPozzMinimalNetworks}.
\begin{enumerate}
    \item Analyticity of $\boldsymbol{\rm G}$ easily follows as in \cite[Lemma 3.11]{PludaPozzMinimalNetworks}.
    
    \item The expression for the first variation $\delta \boldsymbol{\rm G}$ and the fact that it is $Z^\star$-valued follow from \cref{prop:FirstVariation}, while its analyticity follows as in \cite[Lemma 3.11]{PludaPozzMinimalNetworks}.
    
    \item The expression for the second variation $\delta \boldsymbol{\rm G}_0$ at zero and the fact that it is $Z^\star$-valued follow by polarizing the formula in \cref{prop:SecondVariation}.\\
    Arguing as in \cite[Lemma 3.12]{PludaPozzMinimalNetworks} one easily checks that the linear and continuous operator $\mathscr{F}: V \to Z^\star$ given by
    \begin{equation*}
    \mathscr{F}(\overline{X}) \eqdef {\rm I} \left( 
    \bigg((-1)^{1+e^{\ell}_m}\bigg(\partial_s X^\ell(e^\ell_m) - \frac13 \sum_{\ell \in I_m} \partial_s X^\ell(e^\ell_m) \bigg)\bigg),
    \bigg(
    -\partial^2_s X^i +  X^i \bigg)\right),
    \end{equation*}
    where ${\rm I}:Z\to Z^\star$ is the natural isometry, is invertible, and thus it is Fredholm of index $0$. Recalling that Fredholmness is stable with respect to addition of compact operators, we consider
    \begin{equation*}
        V \ni \overline{X} 
        \quad\overset{\mathscr{F}_1}{\mapsto}\quad
          \left( \bigg( 
    (-1)^{e^{\ell}_m} \frac13 \sum_{\ell \in I_m} \partial_s X^\ell(e^\ell_m) \bigg),
    0\right) \in Z^\star ,
    \end{equation*}
    and
    \begin{equation*}
        V \ni \overline{X} 
        \quad\overset{\mathscr{F}_2}{\mapsto}\quad
          \left( \bigg( 
         (-1)^{1+e^{\ell}_m}\sum_{l \in I_m} \omega^l_m(X^l(e^l_m))
    \bigg),
    \bigg(
    -2X^i + \sum_j \Omega_{ij}(\overline{X}) \bigg)
    \right) \in Z^\star ,
    \end{equation*}
    where vectors on the right hand sides above represents elements of $Z^\star$ acting on elements of $Z$ by the obvious pairing. Hence $\mathscr{F}_1$, $\mathscr{F}_2$ are compact, and then $\delta^2 \boldsymbol{\rm G}_0= \mathscr{F} - \mathscr{F}_1 + \mathscr{F}_2$ is Fredholm of index zero.
\end{enumerate}
\end{proof}

As a consequence of the previous lemma, we deduce the desired \L ojasiewicz--Simon inequality for the functional $\G$.

\begin{cor}[\L ojasiewicz--Simon inequality at compact shrinkers]\label{cor:LojaShrinker}
Let $\Gamma_*:G\to\R^2$ be a shrinker, where $G$ is a graph without endpoints. Then there exist $C_{\rm LS},\sigma >0$ and $\theta \in (0,\tfrac12]$ such that the following holds.

If $\Gamma:G\to\R^2$ is a regular network of class $H^2$ such that
\begin{equation*}
     \left(\sum_i \| \gamma^i_* - \gamma^i \|_{H^2(\d x)}^2 \right)^{\frac12} \le \sigma
\end{equation*}
then
 \begin{equation}\label{eq:LojaNetworkMinimali}
     \left|\G(\Gamma)-\G(\Gamma_*) \right|^{1-\theta}
     \le C_{\rm LS} \left( \sum_i \int_{\gamma^i} |(\gamma^i)^\perp + \boldsymbol{k}^i|^2 e^{-|\gamma^i|^2/2} \right)^{\frac12}.
\end{equation}
\end{cor}

\begin{proof}
The proof immediately follows combining \cref{prop:ParametrizzazioneTNtempo} with $\Gamma_t=\Gamma$ for any $t$ and \cref{lem:SuffConditionsLoja}, applying \cite[Corollary 2.6]{PozzettaLoja} which yields the desired \L ojasiewicz--Simon inequality as a consequence of the properties proved in \cref{lem:SuffConditionsLoja}. For further details we refer to the strategy explained in \cite[Corollary 3.15]{PludaPozzMinimalNetworks}.
\end{proof}

\subsection{Convergence of the rescaled motion by curvature}

In this section we prove our main result \cref{thm:Uniqueness}, which is, in turn, implied by the following theorem stating that subconvergence of a rescaled motion by curvature improves to full convergence.

\begin{teo}\label{thm:Convergence}
Let $\widetilde{\Gamma}_\t$ be a rescaled motion by curvature at the origin $0$ and maximal time of existence $T$ of a motion by curvature $\Gamma_t:G\to\R^2$, where $G$ is a graph without endpoints.\\
Suppose that there is a sequence $\t_n\to+\infty$ such that $\widetilde{\Gamma}_{\t_n}$ converges to a regular network $\Gamma_*:G\to\R^2$ in $C^k$ on $G$ for any $k \in \N$, up to reparametrizations.

Then there exists the full limit $\lim_{\t\to+\infty} \widetilde{\Gamma}_\t = \Gamma_*$ in $C^k$ for any $k$, up to reparametrizations.
\end{teo}

\begin{proof}
We first state some interpolation inequalities. For any $k \in \N$ with $k\ge 1$ there exist $\lambda_k>0,\zeta_k\in(0,1)$ such that
\begin{equation}\label{eq:InterpolationSobolevGenerale}
    \|u \|_{H^k(\d x)} \le \lambda_k \| u \|_{L^2(\d x)}^{\zeta_k} \| u \|_{H^{k+1}(\d x)}^{1-\zeta_k},
\end{equation}
for any $u \in H^{k+1}\left((0,1);\R^N\right)$. We shall drop the subscript $k$ when $k=2$.

\medskip

Next we observe that for any $k\in \N$ with $k\ge 3$ and $r>0$ there is $C_k=C_k(r,\Gamma_*) >0$ such that whenever $\widehat\Gamma_\t:G\to \R^2$ is a rescaled motion by curvature with parametrizations $\widehat{\gamma}^i_\t = \gamma^i_* + \NN^i_\t\nu^i_*+ \TT^i_\t\tau^i_*$, i.e., there holds $(\partial_\t \widehat\gamma^i_\t)^\perp = (\widehat\gamma^i_\t)^\perp + \widehat{\boldsymbol{k}}^i_\t$, for $\overline{\NN}_\t \in B_r(0)\subset V$ in the notation of \eqref{eq:DefG}, where the $\TT^i_\t$'s are adapted, for $\t \in [\tau_1,\tau_2]$ with $\tau_1>1$, then
\begin{equation}\label{eq:BoundCk}
    \|\overline{\NN}_\t\|_{H^k(\d x)} \le C_k,
\end{equation}
for $\t \in [\tau_1,\tau_2]$.

Indeed, by \cref{rem:EquivalenzaTangentRescaled} the motion $\widehat\Gamma_\t$ corresponds to a parabolically rescaled motion by curvature $\Gamma^\mu_t$ such that $\Gamma^\mu_{-\frac12}= \widehat\Gamma_{\tau_1}$ via the identities
\begin{equation*}
    \Gamma^\mu_t = \sqrt{-2t}\,\widehat{\Gamma}_{\log(\mu/\sqrt{-t})}
    \qquad
    \mu^2 = \frac{e^{2\tau_{1}}}{2}.
\end{equation*}
Let $\overline{t}\in(-1/2,0)$ be such that $\tau_2= \log\left(\frac{\mu}{\sqrt{-\overline{t}}} \right)$. Assume first that $\overline{t}\le-\tfrac14$. Then $\Gamma^\mu_t$ is a motion by curvature for $t \in [-1/2,\overline{t}]$ with curvature uniformly bounded in $L^2(\d s)$ and with length of the edges uniformly bounded from below away from zero. Recalling \cite[Proposition 5.8]{mannovplusch} and since $\sqrt{-2t} \in [1/\sqrt{2},1]$ for $t \in [-1/2,\overline{t}]$, it is readily checked that this implies \eqref{eq:BoundCk} (we refer the reader to the argument in the proof of \cite[Theorem 5.2]{PludaPozzMinimalNetworks}). In case $\overline{t}>-\tfrac14$, then the previous argument implies estimates \eqref{eq:BoundCk} for $\t \in [\tau_1, \log(\mu/\sqrt{1/4})]\subset [\tau_1,\tau_2]$. Hence the argument can be iterated with $\log(\mu/\sqrt{1/4})$ in place of $\tau_1$. Iterating again and again finitely many times we get that \eqref{eq:BoundCk} holds on $[\tau_1,\tau_2]$.

\medskip

The proof continues adapting the strategy of \cite[Theorem 5.2]{PludaPozzMinimalNetworks}.
Without loss of generality we can assume that $\G(\widetilde\Gamma_\t)>\G(\Gamma_*)$ for any $\t$, for otherwise the rescaled motion by curvature would be stationary, and the thesis would follow trivially.

Let $\sigma,\theta,C_{\rm LS}$ be as in \cref{cor:LojaShrinker}.
For $r_{\Gamma_*}>0$ suitably small, whenever $\widehat\gamma^i\eqdef \gamma^i_*+ \NN^i\nu^i_*+\TT^i\tau^i_*$ is a smooth regular network, for $\overline{\NN}\in B_{r_{\Gamma_*}}(0)\subset V$ in the notation of \eqref{eq:DefG}, where the $\TT^i$'s are adapted, then
\begin{enumerate}[label={\normalfont{\color{blue}\arabic*)}}]

    \item \label{it:r1} there is $r_\sigma \in (0,r_{\Gamma_*})$ such that if $\overline{N}\in B_{r_\sigma}(0)\subset V$, then $\left(\sum_i \| \gamma^i_* - \widehat\gamma^i \|_{H^2(\d x)}^2 \right)^{\frac12} \le \sigma$;
    
    \item \label{it:r1bis} there is $e_*>0$ such that $e^{-|\widehat\gamma^i|^2/2}\ge e_*^2$ for any $i$;
    
    \item \label{it:r2} there exists a constant $C_G>0$, depending only on the graph $G$ and $\Gamma_*$, such that
    \begin{equation*}
    \begin{split}
        &\scal{\widehat\nu^i , \nu^i_*}\ge \frac34 , \qquad\qquad |\scal{\widehat\nu^i , \tau^i_*}| < \frac{1}{C_G}, \\
        &\sum_{m \in J_G}\sum_{\ell \in I_m} \left| a^\ell (x) \scal{\widehat\nu^\ell , \nu^\ell_*}
        +
        \scal{\widehat\nu^\ell , \tau^\ell_*}
        \chi(x)\mathscr{L}^\ell_m (a^{i_m}(x),a^{j_m}(x) )
        \right|^2
        \ge \frac{1}{C_G}\sum_i |a^i(x)|^2
        ,
    \end{split}
    \end{equation*}
    where $\widehat\nu^i$ is the normal vector of $\widehat\gamma^i$, and $i_m$ (resp. $j_m$) denotes the minimal (resp. intermediate) element of $I_m$, for any continuous functions $a^1,\ldots,a^N$;
    
    \item \label{it:r3} there exist $c_1,c_2>0$ such that
    \begin{equation*}
        c_1\le |\partial_x \widehat\gamma^i|^{-1} \le c_2,
    \end{equation*}
    for any $i$.
\end{enumerate}
For any $k\in \N$ with $k\ge3$ let $C_k$ be given by \eqref{eq:BoundCk} with $r=r_\sigma$.\\
By assumptions and by \cref{prop:ParametrizzazioneTNtempo} there exist times $\t_2>\t_1>0$ and $\NN^i_\t,\TT^i_\t,\varphi^i_\t$ as in \cref{prop:ParametrizzazioneTNtempo} such that
\begin{equation}\label{eq:zzxxz}
    ( \G(\widetilde\Gamma_{\t_1})-\G(\Gamma_*) )^\theta < \frac{\theta e_*}{(100 \lambda C_3)^{\frac{1}{\zeta}} C_{\rm LS} \sqrt{C_G\, c_2}} r_\sigma^{\frac{1}{\zeta}}
\end{equation}
\begin{equation}\label{eq:zzxx}
    \begin{split}
    \widetilde\gamma^i_\t\circ \sigma^i \circ \varphi^i_\t = \gamma^i_* + \NN^i_\t \nu^i_* + \TT^i_\t \tau^i_* \eqqcolon \widehat\gamma^i_\t,
    \qquad
    \overline{\NN}_\t\in B_{\frac{r_\sigma}{2}}(0)\subset V,
    \end{split}
\end{equation}
for any $\t \in [\t_1,\t_2]$. Let $S \in [\t_2,+\infty]$ be the supremum of times $s$ such that \eqref{eq:zzxx} holds on $\t \in [\t_1,s]$ for some $\NN^i_\t,\TT^i_\t,\varphi^i_\t$ as in \cref{prop:ParametrizzazioneTNtempo}.

Letting $H(\t)\eqdef (\G(\widetilde\Gamma_\t)-\G(\Gamma_*))^\theta>0$, we can differentiate $H$ to get
\begin{equation*}
    \begin{split}
        -\frac{\d}{\d \t} H(\t) 
        &= \theta H^{\frac{\theta-1}{\theta}}(\t) \sum_i \int_{\widetilde\gamma^i_\t} |(\widetilde\gamma^i_\t)^\perp + \widetilde{\boldsymbol{k}}^i_\t |^2 e^{-|\widetilde\gamma^i_\t|^2/2} 
        \\&= \theta H^{\frac{\theta-1}{\theta}}(\t) \bigg( \sum_i \int_{\widehat\gamma^i_\t} |(\widehat\gamma^i_\t)^\perp + \widehat{\boldsymbol{k}}^i_\t |^2 e^{-|\widehat\gamma^i_\t|^2/2} \bigg)^{\frac12} \left\| (\partial_\t\widehat\Gamma_\t)^\perp\right\|_{L^2(\widehat\mu_t)} \\
        & \overset{\ref{it:r1}}{\ge} \frac{\theta}{C_{\rm LS}}  \left\| (\partial_\t\widehat\Gamma_\t)^\perp\right\|_{L^2(\widehat\mu_t)},
    \end{split}
\end{equation*}
for any $\t \in [\t_1,S)$, where we denoted $ \left\| (\partial_\t\widehat\Gamma_\t)^\perp\right\|_{L^2(\widehat\mu_\t)}^2 \eqdef \sum_i  \int |(\partial_t\widehat\gamma^i_\t)^\perp|^2 \de \widehat\mu^i_\t $, and $\widehat\mu^i_\t\eqdef e^{-|\widehat\gamma^i_\t|^2/2}|\partial_x\widehat\gamma^i_\t|\de x$, and we applied \cref{cor:LojaShrinker} thanks to \ref{it:r1}.
Exploiting \eqref{eq:zzxx}, the above estimate becomes
\begin{equation*}
    \begin{split}
        -\frac{\d}{\d \t} H 
        &\ge \frac{\theta}{C_{\rm LS}}\left( \sum_i  \int_{\widehat\gamma^i_\t} |(\partial_\t\widehat\gamma^i_\t)^\perp|^2 e^{-|\widehat\gamma^i_\t|^2/2} \right)^{\frac12} 
        \overset{\ref{it:r1bis}}{\ge}
        \frac{\theta e_*}{C_{\rm LS}}\left( \sum_i  \int_{\widehat\gamma^i_\t} |\partial_\t \NN^i_\t \scal{\nu^i_\t,\nu^i_*} + \partial_\t\TT^i_\t \scal{\nu^i_\t,\tau^i_*} |^2   \right)^{\frac12} \\
        &\ge \frac{\theta e_*}{C_{\rm LS}}\left( \frac12\sum_{m \in J_G} \sum_{\ell \in I_m}  \int_{\widehat\gamma^\ell_\t} |\partial_\t \NN^\ell_\t \scal{\nu^\ell_\t,\nu^\ell_*} + \partial_\t\TT^\ell_\t \scal{\nu^\ell_\t,\tau^\ell_*} |^2  \right)^{\frac12} \\
        &\overset{\ref{it:r2}}{\ge}  \frac{\theta e_*}{C_{\rm LS}\sqrt{C_G}} \left( \sum_i \int_{\widehat\gamma^i_\t} |\partial_\t \NN^i_\t|^2 \right)^{\frac12} \\
        &\overset{\ref{it:r3}}{\ge}  \frac{\theta e_*}{C_{\rm LS}\sqrt{C_G\, c_2}} \left( \sum_i \int_0^1 |\partial_\t \NN^i_\t|^2 \de x \right)^{\frac12},
    \end{split}
\end{equation*}
for any $\t \in [\t_1,S)$. Hence
\begin{align}
    \left\|\overline{\NN}_{s_2} - \overline{\NN}_{s_1}\right\|_{L^2(\d x)} 
        &\le \int_{s_1}^{s_2} \left\|  \partial_\t \overline{\NN}_\t  \right\|_{L^2(\d x)} \de t \le \frac{C_{\rm LS}\sqrt{C_G\, c_2}}{\theta e_*} H(\t_1) 
        \label{eq:Cauchy1}
        \\ & \overset{\eqref{eq:zzxxz}}{\le}\frac{r_\sigma^{\frac{1}{\zeta}}}{(100 \lambda C_3)^{\frac{1}{\zeta}}},
        \label{eq:Cauchy2}
\end{align}
for any $\t_1\le s_1< s_2 < S$.
By \eqref{eq:InterpolationSobolevGenerale} with $k=2$, and by \eqref{eq:BoundCk} we obtain
\begin{equation*}
    \left\|\overline{\NN}_{s_2} - \overline{\NN}_{s_1}\right\|_{H^2(\d x)} \le \frac{r_\sigma}{50},
\end{equation*}
for any $\t_1\le s_1< s_2 < S$. Since $\overline{\NN}_{\t_1} \in B_{\frac{r_\sigma}{2}}(0)\subset V$, arguing by contradiction one gets that $S=+\infty$. Since $H(\t)\to 0$ as $\t\to+\infty$, then \eqref{eq:Cauchy1} implies that $\overline{\NN}_\t$ is Cauchy in $L^2(\d x)$ for $\t\to+\infty$. Interpolating as before, we deduce that there exists the full limit $\lim_{\t\to+\infty} \widetilde{\Gamma}_\t = \Gamma_*$ in $C^k$ for any $k$.
\end{proof}

\begin{proof}[Proof of \cref{thm:Uniqueness}]
The statement follows from \cref{thm:Convergence} together with the equivalence recalled in \cref{rem:EquivalenzaTangentRescaled}.
\end{proof}

\bibliographystyle{plain}
\addcontentsline{toc}{section}{References}
\bibliography{references} 

\begin{thebibliography}{10}

\bibitem{BaldiHausMan}
P.~Baldi, E.~Haus, and C.~Mantegazza.
\newblock Non-existence of {$theta$}-shaped self-similarly shrinking networks
  moving by curvature.
\newblock {\em Comm. Partial Differential Equations}, 43(3):403--427, 2018.

\bibitem{Bronsardreitich}
L.~Bronsard and F.~Reitich.
\newblock On three-phase boundary motion and the singular limit of a
  vector-valued {G}inzburg-{L}andau equation.
\newblock {\em Arch. Rational Mech. Anal.}, 124(4):355--379, 1993.

\bibitem{Chang2}
J.-E. Chang.
\newblock Stability of regular shrinkers in the network flow.
\newblock 2021, arXiv:2107.04338.

\bibitem{Chang}
J.-E. Chang and Y.-K. Lue.
\newblock Uniqueness of regular shrinkers with two enclosed regions.
\newblock {\em Geom. Dedicata}, 216(1):17, 2022.

\bibitem{Ch03}
R.~Chill.
\newblock On the {{\L}{}}ojasiewicz--{S}imon gradient inequality.
\newblock {\em J. Funct. Anal.}, 201(2):572--601, 2003.

\bibitem{ChFaSc09}
R.~Chill, E.~Fa\v{s}angov\'{a}, and R.~Sch\"{a}tzle.
\newblock Willmore blowups are never compact.
\newblock {\em Duke Math. J.}, 147(2):345--376, 2009.

\bibitem{ChodoshSchulze}
O.~Chodosh and F.~Schulze.
\newblock Uniqueness of asymptotically conical tangent flows.
\newblock {\em Duke Math. J.}, 170(16):3601--3657, 2021.

\bibitem{ColdingMinicozziUniqueness}
T.~H. Colding and W.~P. Minicozzi~II.
\newblock Uniqueness of blowups and {{\L}{}}ojasiewicz inequalities.
\newblock {\em Ann. of Math.}, 182(1):221--285, 2015.

\bibitem{GMP}
M.~G\"{o}{\ss}wein, J.~Menzel, and A.~Pluda.
\newblock On the existence and uniqueness of the motion by curvature of regular
  networks.
\newblock 2020, To appear: Interfaces Free Bound. arXiv:2003.09962.

\bibitem{IlmNevSch}
T.~Ilmanen, A.~Neves, and F.~Schulze.
\newblock On short time existence for the planar network flow.
\newblock {\em J. Differential Geom.}, 111(1):39--89, 2019.

\bibitem{liramazplusae}
J.~Lira, R.~Mazzeo, A.~Pluda, and M.~Saez.
\newblock Short-time existence for the network flow.
\newblock 2021, To appear: Comm. Pure Appl. Math., arXiv:2101.04302.

\bibitem{MaMaNo16}
A.~Magni, C.~Mantegazza, and M.~Novaga.
\newblock Motion by curvature of planar networks, {II}.
\newblock {\em Ann. Sc. Norm. Super. Pisa Cl. Sci. (5)}, 15:117--144, 2016.

\bibitem{MaNoPl21}
C.~Mantegazza, M.~Novaga, and A.~Pluda.
\newblock Type-0 singularities in the network flow - evolution of trees.
\newblock 2021.

\bibitem{mannovplusch}
C.~Mantegazza, M.~Novaga, A.~Pluda, and F.~Schulze.
\newblock Evolution of networks with multiple junctions.
\newblock 2016. arXiv:1611.08254.

\bibitem{MantegazzaNovagaTortorelli}
C.~Mantegazza, M.~Novaga, and V.~M. Tortorelli.
\newblock {Motion by curvature of planar networks}.
\newblock {\em {Ann. Sc. Norm. Super. Pisa, Cl. Sci. (5)}}, 3(2):235--324,
  2004.

\bibitem{PludaPozzMinimalNetworks}
A.~Pluda and M.~Pozzetta.
\newblock {{\L}ojasiewicz--Simon inequalities for minimal networks: stability
  and convergence}, 2022.
\newblock arXiv:2204.08921.

\bibitem{PozzettaLoja}
M.~Pozzetta.
\newblock Convergence of elastic flows of curves into manifolds.
\newblock {\em Nonlinear Analysis}, 214:112581, 2022.

\bibitem{Schulze}
F.~Schulze.
\newblock Uniqueness of compact tangent flows in {M}ean {C}urvature {F}low.
\newblock {\em J. Reine Angew. Math.}, 690:163--172, 2014.

\bibitem{Si83}
L.~Simon.
\newblock Asymptotics for a class of nonlinear evolution equations, with
  applications to geometric problems.
\newblock {\em Ann. of Math. (2)}, 118(3):525--571, 1983.

\end{thebibliography}
\end{document}